\documentclass[9pt,letterpaper]{amsart}
\usepackage{amsmath}
\usepackage{amsfonts}
\usepackage{graphicx}
\usepackage{latexsym, amscd, euscript}
\usepackage{color}
\usepackage{tikz}
\usepackage{enumitem}
\usepackage{mathrsfs}

\usepackage{amssymb}
\usepackage{epsfig}

\usepackage[utf8]{inputenc}
\usepackage[english]{babel}
 
\usepackage{amsthm}

\begin{document}
\theoremstyle{plain}
 \newtheorem{lemma}{Lemma}[section]
\newtheorem{theorem}[lemma]
{Theorem }
\newtheorem{corollary}[lemma]
{Corollary}
\newtheorem{proposition}[lemma]{Proposition}
\newtheorem*{thma}{Theorem 1.1'}

\theoremstyle{definition}
\newtheorem{definition}[lemma]{Definition }
\newtheorem{rmkdef}[lemma]{Remark-Definition}
\newtheorem{conjecture}[lemma]
{Conjecture }
\newtheorem{example}[lemma]
{Example }
\newtheorem{fact}[lemma]{Fact}
\newtheorem{remark}[lemma]{Remark}

\newcommand{\C}{{\mathbb C}}
\newcommand{\e}{\varepsilon}
\renewcommand{\l}{\lambda}
\newcommand{\In}{\operatorname{In}}
\newcommand{\Supp}{\operatorname{Supp}}
\newcommand{\s}{\sigma}
\renewcommand{\t}{\tau}
\newcommand{\NP}{\operatorname{NP}}
\newcommand{\V}{\mathcal V}
\renewcommand{\o}{\omega}
\renewcommand{\k}{\Bbbk}
\newcommand{\Sc}{\mathscr{S}_\preceq}
\newcommand{\T}{\mathcal{T}_\preceq}
\newcommand{\Ker}{\text{Ker}}
\newcommand{\A}{\k[[x]]}
\newcommand{\R}{\mathbb R}
\newcommand{\cha}{\operatorname{char}}
\newcommand{\m}{\mathfrak m}
\newcommand{\g}{\gamma}
\renewcommand{\O}{\mathcal{O}}
\newcommand{\D}{\Delta}
\renewcommand{\P}{\mathbb{P}}
\newcommand{\PP}{\mathcal P}
\renewcommand{\a}{\alpha}
\newcommand{\ovl}{\overline}
\newcommand{\G}{\mathcal{G}_\preceq}
\renewcommand{\(}{(\!(}
\renewcommand{\)}{)\!)}
\renewcommand{\b}{\beta}
\renewcommand{\int}{\operatorname{Int}}
\newcommand{\Int}{\operatorname{Int_{rel}}}
\newcommand{\ord}{\operatorname{ord}}
\newcommand{\K}{{\mathbb K}}
\newcommand{\N}{{\mathbb N}}
\newcommand{\Q}{{\mathbb Q}}
\newcommand{\Z}{{\mathbb Z}}
\newcommand{\sA}{{\mathcal A}}

\setlength\parindent{0pt}

\author{Guillaume Rond} 
\address{Guillaume Rond, Aix-Marseille Universit\'e, CNRS, Centrale Marseille, I2M, UMR 7373, 13453 Marseille, France}

\email{guillaume.rond@univ-amu.fr}

 \author{Bernd Schober}
\address{Bernd Schober\\
Institut f\"ur Algebraische Geometrie\\
Leibniz Universit\"at Hannover\\
Welfengarten 1\\
30167 Hannover\\
Germany}
\email{schober.math@gmail.com}

\title{An irreducibility criterion for  power series}

\begin{abstract}
We prove an irreducibility criterion for polynomials with power series coefficients generalizing previous results given in \cite{GBGP} and \cite{ACLM1}.

\end{abstract}


\thanks{G. Rond was partially supported by ANR projects STAAVF (ANR-2011 BS01 009) and SUSI (ANR-12-JS01-0002-01)}

\maketitle

\section{Introduction}
The aim of this note is to provide a  natural approach to an irreducibility criterion for polynomials with power series coefficients (see Theorem \ref{main}). The first version of the criterion has been given in \cite{GBGP} and then has been generalized in \cite{ACLM1}. In this note we give a more natural and elementary proof of a general version of this criterion. In particular, our statement holds over any field while the previous ones were only proven for algebraically closed fields of characteristic zero. Moreover, the only hypothesis that we need is that the projection of the Newton polyhedron has exactly one vertex while the previous known versions were involving additional technical conditions. \\
Let us recall that the proof given in \cite{GBGP} uses toric geometry and Zariski Main Theorem while the one provided in \cite{ACLM1} is based on a generalization of the Newton's method for plane curves. Our proof is essentially based on the following well known version of Hensel's Lemma:

\begin{proposition}[Hensel's Lemma] Let $(R,\m)$ be a Henselian local ring. A monic polynomial $P(Z)\in R[Z]$, that is the product of two monic coprime polynomials modulo $\m R[Z]$, is in fact the product of two coprime monic polynomials.
\end{proposition}

We begin by giving some definitions and our main result (Theorem \ref{main}). In a second part we give an example showing that our main result cannot be extended in a more general setting.

Finally, let us mention that this irreducibility criterion is very useful in the study of quasi-ordinary hypersurfaces (see \cite{ACLM2} or \cite{MS}).


\section{An irreducibility criterion}
We denote by $\k[[x]]$ the ring of formal power series in $n$ variables $x:=(x_1,\ldots,x_n)$ over a field $\k$. For any vector $\b\in\Z^n$ we set
$$x^\b:=x_1^{\b_1}\cdots x_n^{\b_n}$$
and for any positive integer $q$ 
$$x^q:=x_1^q\cdots x_n^q.$$
Let $P(Z)\in \A[Z]$ be a monic polynomial with coefficients in $\A$. We denote by $\NP(P)$ the Newton polyhedron of $P(Z)$. Let us write
$$P(Z)=Z^d+\sum_{\a\in\Z_{\geq 0}^n,j< d} c_{\a,j}x^\a Z^j.$$
In this note we assume that $P(Z)\neq Z^d$.
The \emph{associated polyhedron} of $P$, denoted by $\D_P$, is the convex hull of
$$\left\{\frac{d\a}{d-j}\mid c_{\a,j}\neq 0\right\}+\R_{\geq 0}^n.$$
Note that $\D_P$ is the projection of $\NP(P)$ from the point $ (0, \ldots, 0, d )$ on the subspace given by the first $ n $ coordinates.

\begin{definition}\label{defi1}
Let $\o\in\R_{>0}^n$. For a non zero element $\displaystyle b=\sum_{\a\in\Z_{\geq 0}^n}b_\a x^\a$ of $\A$ we set 
$$\nu_\o(b):=\min\{\a\cdot\o = \sum_{i=1}^n \a_i \o_i \mid b_\a\neq 0\} \in \R_{\geq 0}$$
$$\text{ and } \In_\o(b):=\sum_{\a\mid \a\cdot\o=\nu_\o(b)}b_\a x^\a.$$
For such a $\o$ and $P(Z)\in \A[Z]$ as before we define $\o_{n+1}\in\R_{\geq 0}$ by
$$\o_{n+1}:=\frac{\min\{v\cdot\o\mid v\in \D_P\}}{d}\in\R_{\geq 0}.$$
Then we set $\o':=(\o,\o_{n+1})$ and we define
$$\nu_{\o'}(P):=\min\{\a\cdot\o+j\o_{n+1}\mid c_{\a,j}\neq 0\}=d\o_{n+1}$$
$$\text{ and }\In_{\o'}(P):=Z^d+\sum_{(\a,j)\mid (\a,j)\cdot\o'=\nu_{\o'}(P)}c_{\a,j} x^\a Z^j.$$
This former polynomial is weighted homogeneous for the weights $\o_1$,..., $\o_n$, $\o_{n+1}$.

\end{definition}

 
\begin{definition}
Let $P(Z)\in\A[Z]$ be a monic polynomial of degree $d$ in $Z$. The polynomial $P$ has an \emph{orthant associated polyhedron} if $\Delta_P=d\g+\R^n_{\geq 0}$ for some $\g\in\Q_{\geq 0}^n$.\\
In this case $\In_{\o'}(P)$ does not depend on $\o$ and we denote it by $P_{\In}$, i.e. 
$$P_{\In}(x,Z):=Z^d+\sum_{(\a,j)\mid \frac{\a}{d-j}=\g}c_{\a,j}x^\a Z^j.$$
In this case we define
$$\ovl P(Z):=P_{\In}(1,Z)=Z^d+\sum_{(\a,j)\mid \frac{\a}{d-j}=\g}c_{\a,j}Z^j\in\k[Z].$$
If we write $\g=\frac{\b}{q}$, where $\b\in\Z_{\geq 0}^n$, $q\in\{1,\ldots,d\}$ and $\gcd(\b_1,\ldots,\b_n,q)=1$, we have that

$$x^{d \b}\ovl P(Z)=P_{\In}(x_1^{q},\ldots, x_n^q,x^{\b } Z).$$
\end{definition}

Here is a picture of the Newton polyhedron of a  polynomial having an orthant associated polyhedron with $n=2$  (thick lines represent the edges of the Newton polyhedron) :
 $$ 
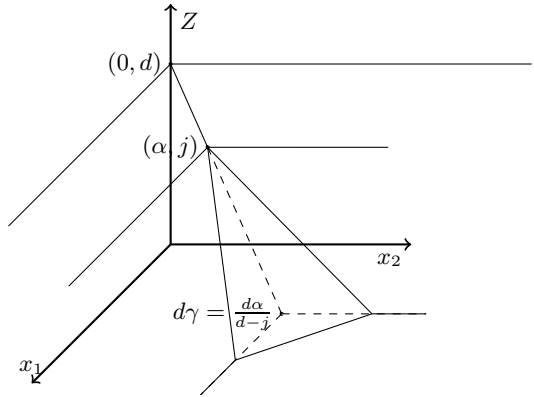
\begin{tikzpicture}[scale=0.8]
  
  \draw[thick,->] (0,0,0) -- (4,0,0) node[anchor=north east]{$x_2$};
\draw[thick,->] (0,0,0) -- (0,4,0) node[anchor=north west]{$Z$};
\draw[thick,->] (0,0,0) -- (0,0,6) node[anchor=south]{$x_1$};

 \filldraw [black] (0,3,0)  coordinate (c) circle (.7pt) node[left] {$(0,d)$};
 
 \draw (c) -- (1,2,1) coordinate (b) circle (.7pt) node[left] {$(\a, j)$}; 
 \draw[dashed] (b) -- (3,0,3) coordinate (a) circle (.7pt) node[left] {$d\g=\frac{d\a}{d-j}$};
 \draw[dashed] (a) -- (3,0,6.5) ;
 \draw[dashed] (a) -- (5.4,0,3) ;
 
 \draw (c) -- (0,3,7) ;
 \draw (c) -- (6,3,0) ;
 
 \draw (b) -- (4,2,1) ;
 \draw (b) -- (1,2,7);
 \draw (b) -- (3,0,5) coordinate (d) ;
 \draw (b) -- (4.5,0,3) coordinate (e);
 \draw (d) -- (e) ;
 \draw (d) -- (3,0,6.5);
 \draw (e) -- (5.4,0,3);

      \end{tikzpicture}  $$


\begin{theorem}\label{main}
Let us assume that $P(Z)$ is irreducible and has an orthant associated polyhedron.
 Then $P_{\In}(x,Z)\in\k[x,Z]$ is not the product of two coprime polynomials. \end{theorem}
 
 \begin{proof}
 Let us assume that $P_{\In}(x,Z)$ is the product of two coprime polynomials of $\k[x,Z]$. We denote by $P_1(x,Z)$ and $P_2(x,Z)$ these two polynomials, so we have
 $$P_{\In}(x,Z)=P_1(x,Z)\cdot P_2(x,Z),$$ 
 and we may assume that they are monic respectively of degree $d_1$ and $d_2$ (with $d_1+d_2=d$) since $P_{\In}(x,Z)$  is monic.
 Let us write $\ovl P_i(Z):=P_i(1,Z)$ for $i=1,2$. Thus we have that
 $$\ovl P(Z)=\ovl P_1(Z)\cdot\ovl P_2(Z).$$
 
 Let $M=M(x,Z):=c\,x^\a Z^j$ be a monomial of $P(x,Z)$. We have that
 $$M(x_1^q,\ldots,x_n^q,x^{\b} Z)=c\,x^{q\a+j\b}Z^j$$
 and $q\a+j\b\geq_\ast d\b$ since $\frac{\a}{d-j}\geq_\ast \frac{\b}{q}$ if $j<d$, where $ \geq_\ast $ denotes the product order on $ \R^n_{\geq 0} $. Thus we have
 \begin{equation}\label{key}P(x_1^q,\ldots,x_n^q,x^{\b} Z)=x^{d\b}\left(\ovl P(Z)+Q(x,Z)\right)\end{equation}
for some $Q(x,Z)\in (x)\A[Z]$. In particular,  $\ovl P(Z)+Q(x,Z)= \ovl P_1(Z)\ovl P_2(Z)$ modulo $(x)$. Thus by Hensel's Lemma 
 $$\ovl P(Z)+Q(x,Z)=\widetilde P_1(x,Z)\cdot\widetilde P_2(x,Z),$$ 
 for some monic polynomials $\widetilde P_1(x,Z)$ and $\widetilde P_2(x,Z)\in\A[Z]$ equal respectively to $\ovl P_1(Z)$ and $\ovl P_2(Z)$ modulo $(x)$.  So we have that
 $$P(x^{q},x^{\b} Z)=\left(x^{d_1\b}\widetilde P_1(x,Z)\right)\cdot\left(x^{d_2\b}\widetilde P_2(x,Z)\right)$$
 and
 $$\left[x^{d_i\b}\widetilde P_i(x,Z)\right]_{\In}=P_i(x^q,x^\b Z) \text{ for } i=1,2.$$
 But we have that
 $$x^{d_i\b}\widetilde P_i(x,Z)=R_i(x,x^{\b}Z)$$
 for some monic polynomials $R_i(x,Z)\in\k[[x]][Z]$ of degree $d_i$.
 Thus
 $$P(x^{q},Z)=R_1(x,Z)\cdot R_2(x,Z)$$
 and ${R_i}_{\In}(x,Z)=P_i(x^q,Z)\in\k[x^q,Z]$ for $i=1,2$.
 Since $P_{\In}=\In_{\o'}(P)$ and ${R_i}_{\In}=\In_{\o'}(R_i)$ for $i=1,2$,  for any $\o\in\R_{>0}^n$ we can apply Lemma \ref{lemma1} for $ P_0=P(x^{q},Z)$ to see that $R_1(x,Z)$, $R_2(x,Z)\in\k[[x^q]][Z]$. Hence $P$ is not irreducible.

 \end{proof}
 
 \begin{remark}
 The key point in the proof of this theorem is the fact that equation \eqref{key} is satisfied when  $P$ has an orthant associated polyhedron.
 \end{remark}
 
  
 \begin{lemma}\label{lemma1}
 Let $P_0\in\k[[x^q]][Z]$ be a monic polynomial, where $q\in\Z_{>0}$, and let us assume that $P_0=R_1R_2$, where $R_1$ and $R_2$ are monic polynomials of $\k[[x]][Z]$. Let $\o\in\R_{>0}^n$ and let $\o'$ be defined as in Definition \ref{defi1}. If $\In_{\o'}(R_1)$, $\In_{\o'}(R_2)\in\k[x^q,Z]$ and if they are coprime then $R_1$, $R_2\in\k[[x^q]][Z]$.
 \end{lemma}
 
 \begin{proof}
  If $\cha(\k)=p>0$ let us write $q=p^em$ with $m\wedge p=1$. If $\cha(\k)=0$ we set $m:=q$ and $p:=1$. Then we define
 $$Q:=\prod R_1(\xi_1 x_1,\ldots, \xi_n x_n, Z)^{p^e}$$
 where $(\xi_1,\ldots,\xi_n)$ runs over the $n$-uples of $m$-th roots of unity in an algebraic closure of $\k$. Then $Q\in\k[[x^q]][Z]$ and $\In_{\o'}(Q)=\In_{\o'}(R_1)^{m^np^e}$. Thus $\In_{\o'}(R_2)$ and $\In_{\o'}(Q)$ are coprime. Hence  the greatest common divisor of $P_0$ and $Q$ in $\k((x))[Z]$ is $R_1$. But the greatest common divisor does not depend of the base field so $R_1$ is also the greatest common divisor of $P_0$ and $Q$ in $\k((x^q))[Z]$ hence $R_1\in\k[[x^q]][Z]$. By symmetry we also get $R_2\in\k[[x^q]][Z]$.
 \end{proof}
 

\begin{corollary}
Let us assume that $P(Z) = Z^d + a_1 Z^{d-1} + \ldots + a_d \in \A[Z]$ is irreducible. Then we have the following properties:

\begin{enumerate}
\item[i)] If $P(Z)$ has an orthant associated polyhedron the convex hull of $\Supp(\In_{\o'}(P))$ is a segment  joining $(0,d)$ to $(d\g,0)$, and $d\g$ is the initial exponent of $a_d$ for the valuation $\nu_\o$.

\item[ii)] If $P(Z)$ has an orthant associated polyhedron let $u\in\Z^{n+1}$ be the primitive vector such that $m u=(-d\g,d)$ for some $m\in\N$, and set $y:=(x,Z)$. Then we can write
$$P_{\In}(x,Z)=x^{d\g}Q(y^u)$$
where $Q(T)\in\k[T]$ is not the product of two coprime polynomials. In particular, $Q(T)$ has only one root in an algebraic closure of $\k$.

\item[iii)] If the Newton polyhedron of $P(Z)$ has no compact face of dimension $>1$ then $P(Z)$ has an orthant associated polyhedron and its Newton polyhedron
  has only one compact face of dimension one which is the segment of i).

\end{enumerate}
\end{corollary}

 \begin{proof}
 
If $P_{\In}(x,0)= 0$ then $Z$ divides $P_{\In}(x,Z)$. But by Theorem \ref{main} $P_{\In}(x,Z)$ is not the product of two coprime polynomials thus $P_{\In}(x,Z)=Z^d$. This contradicts the fact that $P_{\In}(x,Z)$ has a non zero monomial of the form $x^\a Z^j$ for $j<d$. Hence $P_{\In}(x,0)\neq 0$ and $i)$ is proven.\\

We can write 
$$P_{\In}(x,Z)=Z^d+\sum_{j=0}^{d-1}c_{(d-j)\g,j}x^{(d-j)\g} Z^j.$$
 So we have that
$$P_{\In}(x,Z)=x^{d\g}\left(\frac{Z^d}{x^{d\g}}+\sum_{j=0}^{d-1}c_{(d-j)\g,j} \, \frac{Z^j}{x^{j\g}} \right).$$
By $i)$ we have that $d\g\in\Z_{\geq 0}^n$. This implies that $j\g\in\Z_{\geq0}^n$ as soon as $c_{(d-j)\g,j}\neq 0$. For any such $j$, let $i\geq 0$ be such that 
\begin{equation}
\label{iu}
i u=(-j\g,j).
\end{equation}
Then $i \in\Z_{\geq 0}$ since $u$ is primitive.\\
 Thus 
$P_{\In}(x,Z)=x^{d\g}(y^{mu}+\sum_{i<m}c_i\,y^{iu})$, where
$$c_i:=c_{iu+(d\g,0)}\ \ \ \forall i.$$
We set $Q(T):=T^{m}+\sum_{i<m}c_iT^{i}\in\k[T]$. If $Q(T)$ has two distinct roots in an algebraic closure  of $\k$ then $Q(T)$ may be factorized as the product of two coprime monic polynomials, let us say $Q(T)=Q_1(T)\cdot Q_2(T)$, where $Q_1(T)$ and $Q_2(T)\in\k[T]$ are coprime and monic. Let $m_1$ and $m_2$ be the respective degrees of $Q_1$ and $Q_2$ and define $d_i\in\Z_{\geq 0}$ by
$$(-d_i\g,d_i)=m_i u\text{ for } i=1,2.$$
Then we have 
$$P_{\In}(x,Z)=x^{d\g}Q(y^u)=\left(x^{d_1\g}Q_1(y^u)\right)\cdot\left(x^{d_2\g}Q_2(y^u)\right).$$
Moreover, by \eqref{iu}, a monomial of $x^{d_1\g}Q_1(y^u)$ has the form
$$c\, x^{d_1\g}y^{iu}=c\, x^{d_1\g} \left( \frac{Z^j}{x^{j\g}} \right) = c\, x^{(d_1 - j)\g} Z^j,$$
for $ 0 \leq i \leq m_1 $, i.e.~for $ 0 \leq j \leq d_1 $.
Hence $x^{d_1\g}Q_1(y^u)\in\k[x,Z]$. By symmetry we also have that $x^{d_2\g}Q_2(y^u)\in\k[x,Z]$. 

Then the polynomials $P_1(x,Z):=x^{d_1\g}Q_1(y^u)$ and $P_2(x,Z):=x^{d_2\g}Q_2(y^u)$ are coprime which contradicts Theorem \ref{main}. Thus $ii)$ is proven.\\

Let us assume that the Newton polyhedron of $P(Z)$ does not have an orthant associated polyhedron. This means that $\D_P$ has at least two distinct vertices denoted by $\g_1$ and $\g_2$ such that the segment $[\g_1,\g_2]$ is included in the boundary of $\D_P$. Thus the Newton polyhedron of $P$ has at least three different vertices $a:=(0,d)$, $b:=(\frac{d-j}{d}\g_1,j)$ and $c:=(\frac{d-k}{k}\g_2,k)$. 
Since $a$, $b$, $c$ are vertices of $\NP(P)$  the triangle delimitated by these three points is a face of $\NP(P)$ so the Newton polyhedron of $P$ has at least one face of dimension two.

 \end{proof}

\section{An example concerning compact faces of dimension $>1$}

Let $n=2$ and let us replace the variables $( x_1, x_2) $ by $ (x, y) $ for simplicity. We set 
$$ P(Z): = Z^2 - ( x^3 - y^5 )^2 + y^{11} = ( Z - x^3 + y^5)(Z + x^3 - y^5) + y^{11} $$
 seen as a polynomial of $\k[[x,y]][Z]$ where $\k$ is an algebraically closed field of  characteristic different from 2.\\
 We will show that $P$ does not have  an orthant associated polyhedron, since $\D_P$ has two different vertices. On the other hand, we will prove that $P(Z)$ is irreducible while for every $\o\in\R_{>0}^2$ the polynomial $\In_{\o'}(P)$ is always the product of two coprime monic polynomials. This shows that Theorem \ref{main} cannot be extended to polynomials without an orthant associated polyhedron.\\
 
 The Newton polyhedron of $P(Z)$ is the convex hull of
 $$\{(6,0,0), (0,10,0), (0,0,2)\}+\R_{\geq0}^3.$$
 The associated polyhedron $\D_P$ of $P(Z)$ is the convex hull of
 $$\{(6,0), (0,10)\}+\R_{\geq0}^2$$
 and has two vertices  $ v = ( 6, 0 ) $ and $ u = (0,10) $.
For $\o\in\R_{>0}^2$, if $6\o_1<10\o_2$ then
$$\In_{\o'}(P)=Z^2-x^6=(Z-x^3)(Z+x^3).$$
If  $6\o_1>10\o_2$ then we have that 
$$\In_{\o'}(P)=Z^2-y^{10}=(Z-y^5)(Z+y^5).$$
If $6\o_1=10\o_2$ we have that
$$\In_{\o'}(P)=Z^2-(x^3-y^5)^2=(Z - x^3 + y^5)(Z + x^3 - y^5).$$
Thus in all cases $\In_{\o'}(P)$ is the product of two coprime polynomials (since $\cha(\k)\neq 2$). 

On the other hand, $P(Z)$ is irreducible since $( x^3 - y^5 )^2 - y^{11} $ is not a square in $\k[[x,y]]$.


\begin{thebibliography}{00}

\bibitem[ACLM1]{ACLM1} E. Artal Bartolo, P. Cassou-Nogu\`es, I.  Luengo, A. Melle Hern\'andez,  On $\nu$-quasi-ordinary power series: factorization, Newton trees and resultants, \textit{Topology of algebraic varieties and singularities} (Jaca, 2009), Contemp. Math., vol. 538, Amer. Math. Soc., Providence, RI, 2011, pp. 321-343.

\bibitem[ACLM2]{ACLM2} E. Artal Bartolo, P. Cassou-Nogu\`es, I.  Luengo, A. Melle Hern\'andez, Quasi-ordinary singularities and Newton trees, \textit{Mosc. Math. J.}, \textbf{13}, (2013), no. 3, 365-398. 

\bibitem[GBGP]{GBGP} E. R. Garc\'ia Barroso, P. D. Gonz\'alez-P\'erez, Decomposition in bunches of the critical locus of a quasi-ordinary map, \textit{Compos. Math.}, \textbf{141}, (2005), no. 2, 461-486. 

\bibitem[MS]{MS} H. Mourtada, B. Schober, A polyhedral characterization of quasi-ordinary singularities, Arxiv:1512.07507.

\end{thebibliography}
\end{document}